\documentclass[a4paper,11pt]{amsart}
\usepackage{CJK}
\usepackage{amsmath}
\usepackage{amsthm,amssymb,latexsym}
\usepackage{amsmath,calligra}
\usepackage{url,ae}
\usepackage{t1enc}
\usepackage{eufrak}
\usepackage{amsfonts}
\usepackage{graphicx}
\usepackage{hyperref}
\usepackage{amscd}
\usepackage{mathrsfs}
\usepackage{bbm}
\usepackage{color}
\usepackage{txfonts}

\usepackage{multirow,makecell}
\usepackage[format=hang,font=small,textfont=it]{caption}

\newtheorem{theorem}{Theorem}[section]
\newtheorem{corollary}[theorem]{Corollary}
\newtheorem{lemma}[theorem]{Lemma}
\newtheorem{proposition}[theorem]{Proposition}

\newtheorem{question}{Question}[section]

\theoremstyle{definition}
\newtheorem{definition}[theorem]{Definition}
\theoremstyle{remark}
\newtheorem{remark}[theorem]{Remark}
\theoremstyle{example}

\newcommand{\be}{\begin{equation}}
\newcommand{\ee}{\end{equation}}
\newcommand{\bea}{\begin{eqnarray}}
\newcommand{\eea}{\end{eqnarray}}
\newcommand{\ben}{\begin{eqnarray*}}
\newcommand{\een}{\end{eqnarray*}}
\newcommand{\bet}{\begin{equation}
\begin{split}}
\newcommand{\eet}{\end{split}
\end{equation}}

\DeclareMathOperator{\Ext}{\mathscr{E}\text{\kern -3pt {\calligra\Large xt}}\,\,}

\begin{document}

\title[Grauert-Riemenschneider multiplier ideals and the Brian\c{c}on-Skoda number]
      {Grauert-Riemenschneider multiplier ideal sheaves and the (optimal) Brian\c{c}on-Skoda number}

\author{Zhenqian Li}

\date{\today}
\thanks{E-mail: lizhenqian@amss.ac.cn}

\begin{abstract}
The goal of this note is to survey some recent results on the Grauert-Riemenschneider multiplier ideal sheaves on any (reduced) complex space of pure dimension. In particular, we obtain the Brian\c{c}on-Skoda number for any Noetherian ring of weakly holomorphic functions with weakly rational singularities (\emph{not} essentially of finite type over $\mathbb{C}$ and Cohen-Macaulay local rings), which will partially answer a question of Huneke.
\end{abstract}

\maketitle

\section{Introduction}

In \cite{Li_BS}, the author introduced a notion of Grauert-Riemenschneider multiplier ideal sheaves on any (reduced) complex space of pure dimension (\emph{not} necessarily normal or $\mathbb{Q}$-Gorenstein), motivated by the definition of Grauert-Riemenschneider canonical sheaf. Relying on the new multiplier ideals, the author established the original Brian\c{c}on-Skoda theorem for analytic local rings with normal weakly rational singularities. In the present note, we are going to investigate the Grauert-Riemenschneider multiplier ideals further and prove some related properties (e.g., the weak holomorphicity and Skoda-type division theorem, etc.). As applications, following the idea adopted in \cite{Li_BS}, we also obtain the Brian\c{c}on-Skoda number for Noetherian rings of \emph{weakly holomorphic functions} (i.e., germs of locally bounded \emph{meromorphic} functions on singular complex spaces) with weakly rational singularities, and give some characterizations of validity of the original Brian\c{c}on-Skoda theorem.

Throughout this note, all complex spaces are always assumed to be reduced and paracompact, unless otherwise mentioned; we refer to \cite{GR84, Richberg68} for main references on the theory of complex spaces.

\subsection{Grauert-Riemenschneider multiplier ideal sheaves}

Let $X$ be a complex space of pure dimension $n$ and $\varphi\in L_{\text{loc}}^1(X_\text{reg})$ with respect to the Lebesgue measure. Then, we can define the \emph{Grauert-Riemenschneider multiplier ideal sheaf} $\mathscr{I}_\text{GR}(\varphi)\subset\mathscr{M}_X$ associated to the weight $\varphi$ on $X$ (see Definition \ref{MIS_GR}). Moreover, we have the following:

\begin{proposition} \emph{(\cite{Li_BS}, Proposition 3.3).} \label{property}
Let $\varphi\in\emph{Psh}(X)$ be a plurisubharmonic function on $X$ such that $\varphi\not\equiv-\infty$ on every irreducible component of $X$. Then, $\mathscr{I}_\emph{GR}(\varphi)\subset\mathscr{M}_X$ is a coherent fractional ideal sheaf and satisfies the strong openness property, i.e.,      $$\mathscr{I}_\emph{GR}(\varphi)=\bigcup\limits_{\varepsilon>0}\mathscr{I}_\emph{GR}\big((1+\varepsilon)\varphi\big).$$
\end{proposition}

Different from the usual multiplier ideal defined by integrability with respect to the Lebesgue measure (cf. \cite{De10}, Definition 5.4), the Grauert-Riemenschneider multiplier ideal sheaf is in general a \emph{fractional} ideal sheaf, which is an ideal sheaf when $X$ is a normal complex space and agrees with the usual multiplier ideal if $X$ is smooth (cf. \cite{Li_BS}). So it is interesting for us to understand how ``singular'' the multiplier ideals might be. Fortunately, the singularities encoded by Grauert-Riemenschneider multiplier ideals are not too bad and in fact we have the following

\begin{theorem} \label{GR_weakly}
With the same hypotheses as in Proposition \ref{property}, then we obtain that $\mathscr{I}_\emph{GR}(\varphi)\subset\widehat{\mathcal{O}}_X$, the sheaf of weakly holomorphic functions on $X$. In addition, $\mathscr{I}_\emph{GR}(\varphi)$ is integrally closed in $\widehat{\mathcal{O}}_X$.
\end{theorem}

Combining with a relative version of Grauert-Riemenschneider vanishing theorem for the higher direct images, we formulate a Skoda-type division theorem for Grauert-Riemenschneider multiplier ideal sheaves as follows (see also \cite{Li_BS}).

\begin{theorem} \label{SkodaDivision}
Let $X$ be a complex space of pure dimension $n$ and $\varphi\in\emph{Psh}(X)$ such that $\varphi\not\equiv-\infty$ on every irreducible component of $X$. Let $\mathfrak{a}\subset\widehat{\mathcal{O}}_X$ be a nonzero ideal sheaf with $r$ (local) generators. Then, for any integer $k\geq q:=\min\{n,r\}$, we have
$$\mathscr{I}_\emph{GR}\big(\varphi+k\varphi_\mathfrak{a}\big)=\mathfrak{a}\cdot\mathscr{I}_\emph{GR}\big(\varphi+(k-1)\varphi_\mathfrak{a}\big),$$
where $\varphi_{c\cdot\frak{a}}:=\frac{c}{2}\log(\sum_k|g_k|^2)$ for each $c\geq0$ and $(g_k)$ is any local system of generators of $\frak{a}$.
\end{theorem}

\begin{remark}
What we would like to mention is that the weight $\varphi_{c\cdot\frak{a}}\in L_{\text{loc}}^1(X_\text{reg})$ here is not plurisubharmonic (\emph{not} even upper semi-continuous) on $X$ in general. Let $\pi:\widetilde X\to X$ be any resolution of singularities of $X$. Then, we obtain that $\varphi_{c\cdot\frak{a}}\circ\pi$ is plurisubharmonic on $\widetilde X$, which implies that $\varphi_{c\cdot\frak{a}}$ is plurisubharmonic if $\varphi_{c\cdot\frak{a}}$ is upper semi-continuous by Proposition 2.1 in \cite{CM85} (see also \cite{Li_pullback} for more details).
\end{remark}

\subsection{The Brian\c{c}on-Skoda theorem for weakly holomorphic functions}

Let $\mathcal{I}$ be a nonzero ideal of the ring $\mathcal{O}_n=\mathbb{C}\{z_1,...,z_n\}$ of convergent power series in $n$ variables, $g= (g_1,...,g_r)$ a system of generators of $\mathcal{I}$ and $\overline{\mathcal{I}}$ the integral closure of $\mathcal{I}$ in $\mathcal{O}_n$. In \cite{B-S}, Brian\c{c}on and Skoda established the following remarkable result on the comparison of powers of an ideal with the integral closure of their powers, which confirmed a question posed by J. Mather (see \cite{De10}, \S11.B):\\

\textbf{Brian\c{c}on-Skoda Theorem.} If $q=\min\{n,r\}$, then $\overline{\mathcal{I}^{k+q-1}}\subset\mathcal{I}^k,\ k\in\mathbb{N}$. Moreover, the number $q$ is optimal.\\

The original proof given by Brian\c{c}on and Skoda is based on the so-called Skoda's $L^2$ division theorem for holomorphic functions (see \cite{De10}, Theorem 11.13). From a purely algebraic viewpoint, the Brian\c{c}on-Skoda theorem was generalized to arbitrary regular rings by Lipman and Sathaye in \cite{L-S}, and partially extended to pseudo-rational rings by Lipman and Teissier in \cite{L-T}. Later on, Huneke \cite{Huneke92} proved the following similar ``uniform Brian\c{c}on-Skoda'' for much more broad classes of Noetherian rings by the characteristic $p$ method relied on the theory of tight closure:

\begin{theorem} \emph{(\cite{Huneke92}, Theorem 4.13).} \label{Huneke_BS}
Let $R$ be a Noetherian reduced ring. If $R$ satisfies at least one of the following conditions, then there exists a positive integer $\nu$ such that for all ideals $I$ of $R$, $\overline{I^{k+\nu}}\subset I^k,\ k\in\mathbb{N}$.

$(i)$ $R$ is essentially of finite type over an excellent Noetherian local ring.

$(ii)$ $R$ is of characteristic $p$, and $R^{1/p}$ is module finite over $R$.

$(iii)$ $R$ is essentially of finite type over $\mathbb{Z}$.
\end{theorem}

Here the constant $\nu$ must be much larger than $\dim R$ in general and the optimal such $\nu$ is called the \emph{Brian\c{c}on-Skoda number} for the ring $R$; one can refer to \cite{SH06} and \cite{ASS10, Szna10} for more details on the above ``uniform Brian\c{c}on-Skoda'' as well. Specifically, in his paper, Huneke pointed that it is of great interest to find what the Brian\c{c}on-Skoda number is in Theorem \ref{Huneke_BS} in terms of invariants of the ring $R$ and such a number is just $\dim R-1$ for regular $R$ by the original Brian\c{c}on-Skoda theorem (cf. \cite{Huneke92}, Remark 4.14). Thus, if we could establish the original Brian\c{c}on-Skoda theorem for a class of Noetherian rings, we would obtain the Brian\c{c}on-Skoda number $\dim R-1$ immediately.

In order to establish the original Brian\c{c}on-Skoda theorem for ``singular'' Noetherian rings in the analytic setting, a natural idea is to generalize Skoda's $L^2$ division theorem for holomorphic functions to singular complex spaces. Unfortunately, a trivial generalization to the singular case is not possible due to Theorem \ref{Skoda_negative} in section \ref{Non_division}. In \cite{Li_BS}, the author established the original Brian\c{c}on-Skoda theorem for analytic local rings with normal weakly rational singularities (\emph{not} essentially of finite type over $\mathbb{C}$ and Cohen-Macaulay), by introducing the notion of Grauert-Riemenschneider multiplier ideal sheaves. In this section, combining weak holomorphicity of Grauert-Riemenschneider multiplier ideals with some techniques adopted in \cite{Li_BS}, we are able to establish the following:

\begin{theorem} \label{BS_weakly}
Let $X$ be an $n$-dimensional complex space with $x\in X$ a point and $\mathcal{I}\subset\widehat{\mathcal{O}}_{X,x}$ a nonzero ideal with $r$ generators. If $x\in X$ is a weakly rational singularity, then the original Brian\c{c}on-Skoda theorem holds for the Noetherian ring $\widehat{\mathcal{O}}_{X,x}$. In particular, the Brian\c{c}on-Skoda number for $\widehat{\mathcal{O}}_{X,x}$ is $n-1$.
\end{theorem}

\begin{remark}
In \cite{H-V}, an algebraic analogue was established for the case that $R$ is normal local Gorenstein and essentially of finite type over a field of characteristic 0, where the Cohen-Macaulayness of $R$ is necessary.
\end{remark}

\begin{remark}
On the other hand, relying on the Nadel-Ohsawa multiplier ideal sheaves introduced in \cite{Li_multiplier} and a version of Ohsawa-Takegoshi extension theorem presented in \cite{G-Z_optimal}, we can establish the following Brian\c{c}on-Skoda type theorem:

Let $X$ be a complex space of pure dimension $n$ with $x\in X$ a singularity, and $\mathcal{I}\subset\mathcal{O}_{X,x}$ a nonzero ideal with $r$ generators. Set $q=\min\{n,r\}$. Then we have
$$\overline{\mathcal{J}ac_{X,x}\cdot\mathcal{I}^{k+q-1}}\subset\mathcal{I}^k, \ k\in\mathbb{N},$$
where $\mathcal{J}ac_{X}$ is the Jacobian ideal of $X$.
\end{remark}

As a direct consequence of Theorem \ref{BS_weakly}, we can derive the main result in \cite{Li_BS}:

\begin{corollary} \emph{(\cite{Li_BS}, Theorem 1.1).}
The original Brian\c{c}on-Skoda theorem is valid for any analytic local ring with normal weakly rational singularities.
\end{corollary}

Motivated by Theorem \ref{BS_weakly}, a natural question is whether we could characterize weakly rational singularities by the original Brian\c{c}on-Skoda theorem. Concretely, we raise the following

\begin{question} \label{Q_SingBS}
Let $X$ be complex space of pure dimension $n$ with $x\in X$ a point and $\mathcal{I}\subset\widehat{\mathcal{O}}_{X,x}$ a nonzero ideal with $r$ generators.

Could we show that , $x\in X$ is a weakly rational singularity if and only if the original Brian\c{c}on-Skoda theorem holds for the Noetherian ring $\widehat{\mathcal{O}}_{X,x}$?
\end{question}

When $n=1$, we will obtain

\begin{theorem} \label{BS_1D}
Let $X$ be a complex curve with $x\in X$ a point. Then, the original Brian\c{c}on-Skoda theorem holds for the Noetherian ring $\widehat{\mathcal{O}}_{X,x}$. Moreover, the following statements are equivalent:

$(1)$ $x\in X$ is a weakly rational singularity.

$(2)$ $x\in X$ is a regular point.

$(3)$ The original Brian\c{c}on-Skoda theorem holds for the local ring $\mathcal{O}_{X,x}$; namely, $\overline{\mathcal{I}^k}=\mathcal{I}^k$ for any nonzero ideal $\mathcal{I}\subset\mathcal{O}_{X,x}$ and $k\in\mathbb{N}$.
\end{theorem}

When $n=2$, the answer of Question \ref{Q_SingBS} is positive, i.e.,

\begin{theorem} \label{BS_2D}
Let $X$ be a complex space of pure dimension two with $x\in X$ a point and $\mathcal{I}\subset\widehat{\mathcal{O}}_{X,x}$ a nonzero ideal with $r$ generators. Then, $x\in X$ is a weakly rational singularity if and only if the original Brian\c{c}on-Skoda theorem holds for the Noetherian ring $\widehat{\mathcal{O}}_{X,x}$.
\end{theorem}

Due to the Cohen-Macaulayness of normal complex surfaces, as an application of Theorem \ref{BS_2D}, we have

\begin{corollary} \emph{(\cite{Li_BS}, Theorem 1.5).}
Let $X$ be a normal complex surface with $x\in X$ a point. Then, $x\in X$ is a rational singularity if and only if the original Brian\c{c}on-Skoda theorem holds for the local ring $\mathcal{O}_{X,x}$.
\end{corollary}

\begin{remark}
The above corollary does not hold when $\dim X\geq3$; see \cite{A-S84} for some examples of normal weakly rational singularities which are not rational.
\end{remark}

\section{Terminologies and basic facts}

Firstly, let us start by recalling some concepts and notations that will be used throughout this paper. Various classes of the singularities of complex spaces are given exactly as in the algebraic context (cf. \cite{A-S84, K-M98}).

\begin{definition} \label{canonical}
Let $X$ be a complex space of pure dimension $n$. The \emph{Grothendieck canonical sheaf} (or \emph{dualizing sheaf}) $\omega_X$ on $X$ is defined by $$\Gamma(U,\omega_X)=\text{Ext}_{\mathcal{O}_D}^{N-n}(\mathcal{O}_U,\Omega_D^N){\big|}_U,$$ where $U\subset X$ is an open subset isomorphic to a complex model space in a domain $D$ of some $\mathbb{C}^N$. Moreover, $\omega_X$ is independent of the local embedding; and $\omega_X\cong i_*(\omega_{X_\text{reg}})$ for normal $X$, where $i:X_\text{reg}\hookrightarrow X$ is the natural inclusion.

The \emph{Grauert-Riemenschneider canonical sheaf} $\omega_X^{\text{GR}}$ on $X$ is defined by
$$\Gamma(U,\omega_X^{\text{GR}})=\{\sigma\in\Gamma(U\cap X_\text{reg},\omega_{X_\text{reg}})\ |\
(\sqrt{-1})^{n^2}\sigma\wedge\overline\sigma\in L_\text{loc}^1(U)\}$$
for any open subset $U\subset X$. Moreover, it follows that $\omega_X^{\text{GR}}=\pi_*\omega_{\widetilde X}$ for any resolution of singularities $\pi:\widetilde X\to X$ and $\omega_X^{\text{GR}}\subset\omega_X$ (cf. \cite{HenPa99}).
\end{definition}

\begin{remark} (cf. \cite{De10}, Remark 5.17).
Given $\varphi\in L_\text{loc}^1(X_\text{reg})$, we denote by $\omega_X^{\text{GR}}(\varphi)$ the $\mathcal{O}_X$-sheaf on $X$ defined by
$$\Gamma(U,\omega_X^{\text{GR}}(\varphi))=\{\sigma\in\Gamma(U\cap X_\text{reg},\omega_{X_\text{reg}})\ |\
(\sqrt{-1})^{n^2}e^{-2\varphi}\sigma\wedge\overline\sigma\in L_\text{loc}^1(U)\}$$
for any open subset $U\subset X$. In particular, for any $\varphi\in\text{Psh}(X)$, the $\mathcal{O}_X$-sheaf $\omega_X^{\text{GR}}(\varphi)$ satisfies the functoriality property, i.e.,
$$\pi_*\omega_{\widetilde X}^{\text{GR}}(\varphi\circ\pi)=\omega_X^{\text{GR}}(\varphi),$$
where $\pi:\widetilde X\to X$ is any proper modification.
\end{remark}

As a result, we have the following version of Nadel vanishing theorem on complex spaces with singularities (cf. \cite{De10}, Remark 5.17; also \cite{Ohsawa_book18}, Theorem 2.25).

\begin{theorem} \label{NaDe_vanishing}
Let $(X,\omega)$ be a weakly pseudoconvex K\"ahler space of pure dimension $n$. Let $(L,e^{-2\varphi_L})$ be a holomorphic line bundle on $X$ equipped with a singular Hermitian metric of plurisubharmonic weight $\varphi_L$ and $\sqrt{-1}\partial\overline\partial\varphi_L\geq\varepsilon\omega$ for some positive continuous function $\varepsilon$ on $X_\text{\emph{reg}}$. Then,
$$H^q\big(X,\omega_X^{\emph{GR}}(\varphi_L)\otimes\mathcal{O}_X(L)\big)=0$$
for all $q\geq1$.
\end{theorem}

\begin{definition} (cf. \cite{GPR94, Ishii_sing}).
Let $X$ be a complex space of pure dimension. $X$ is said to be \emph{$\mathbb{Q}$-Gorenstein} if there exists a positive integer $r$ such that $\omega_X^{[r]}:=(\omega_{X}^{\otimes r})^{**}$, the bidual of $\omega_{X}^{\otimes r}$, is a locally free sheaf of rank one. We call the minimum of such $r$ the \emph{index of $X$}, and then $X$ is called \emph{$r$-Gorenstein}. $X$ is called \emph{Gorenstein} if it is $1$-Gorenstein and Cohen-Macaulay.

A local generator $\sigma$ of $\omega_X^{[r]}$ defines a holomorphic pluricanonical form without zeros on $X_{\text{reg}}$. In addition, if $X$ is normal, it follows that $\omega_X^{[r]}:=i_*(\omega_{X_\text{reg}}^{\otimes r})$, where $i:X_\text{reg}\to X$ the natural inclusion.
\end{definition}

\begin{remark}
We need point out here that our definition of $\mathbb{Q}$-Gorenstein complex space $X$ does not require the normality of $X$ as usual.
\end{remark}

\begin{definition} (cf. \cite{L-Z}). \label{MIS_QG}
Let $X$ be a $\mathbb{Q}$-Gorenstein complex space of dimension $n$ and $\varphi\in L_\text{loc}^1(X_\text{reg})$ with respect to the Lebesgue measure. The \emph{multiplier ideal sheaf} associated to $\varphi$ on $X$ is defined to be the $\mathcal{O}_X$-submudle $\mathscr{I}(\varphi)\subset\mathscr{M}_X$ of germs of meromorphic functions $f\in\mathscr{M}_{X,x}$ such that
$$|f|^2e^{-2\varphi}\Big((\sqrt{-1})^{rn^2}\sigma\wedge\overline\sigma\Big)^{\frac{1}{r}}$$
is locally integrable at $x$, where $\sigma$ is a local generator of $\omega_X^{[r]}$ at $x$ and $r$ is the index near $x$. Here, $\varphi$ is regarded as the weight function of $\mathscr{I}(\varphi)$. If $\varphi\equiv-\infty$ on $X$, we put $\mathscr{I}(\varphi)=0$.
\end{definition}

\begin{remark}
When $X$ is smooth and $\varphi\in\text{Psh}(X)$, it follows that $\mathscr{I}(\varphi)\subset\mathcal{O}_X$ is nothing but the multiplier ideal sheaf introduced by Nadel (see also \cite{De10}).
\end{remark}

\begin{definition} (cf. \cite{A-S84}).
Let $X$ be a complex space of pure dimension $n$ and $x\in X$ a point. Let $\pi:\widetilde X\to X$ be a resolution of singularities of $X$.

The point $x$ is called a \emph{weakly rational singularity} of $X$ if $\omega_X=\omega_X^{\text{GR}}$ near the point $x$, i.e., every holomorphic section of $\omega_X$ on a neighborhood of $x$ is locally square integrable around $x$.

The point $x$ is called a \emph{rational singularity} of $X$ if one of the following equivalent conditions holds:

$(1)$ $\omega_X=\omega_X^{\text{GR}}$ near the point $x$ and $\mathcal{O}_{X,x}$ is a Cohen-Macaulay local ring;

$(2)$ $x$ is a normal point of $X$ and $(R^{q}\pi_*\mathcal{O}_{\widetilde X})_x=0$ for all $q>0$.

We say that $X$ has at most (weakly) rational singularities if each point in $X$ is a (weakly) rational singularity.
\end{definition}

\begin{remark} (cf. \cite{A-S84}). \label{rational}
$(1)$ The above definition is independent of the resolution $\pi$ and every regular point of $X$ is a (weakly) rational singularity.

$(2)$ If $n=1$, then $x$ is a weakly rational singularity iff $x$ is a rational singularity iff $x$ is a regular point of $X$.

$(3)$ If $n\geq2$, then $x$ is a weakly rational singularity iff $(R^{n-1}\pi_*\mathcal{O}_{\widetilde X})_x=0$.
\end{remark}

\begin{proposition} \label{invariance_weakly}
Let $X$ be a complex space of pure dimension $n\ (n\geq2)$ and let $\xi:\widehat X\to X$ be a one-sheeted analytic covering of $X$. Then, $X$ has at most weakly rational singularities if and only if $\widehat X$ has at most weakly rational singularities. In further, if $X$ has at most weakly rational singularities, we have $\omega_{X}=\xi_*\omega_{\widehat X}$.
\end{proposition}

\begin{proof}
Since $\xi$ is a finite holomorphic mapping, then for any $x\in X$, $\xi^{-1}(x)$ has a fundamental system of neighborhoods consisting of open sets of the form $\xi^{-1}(U)$ with open set $U\subset X$ containing $x$. Moreover, if $\hat x_1,...,\hat x_m$ are the distinct points of a fiber $\xi^{-1}(x)$, then it follows that there exists sufficiently small open neighborhood $U$ of $x$ in $X$ such that $\xi^{-1}(U)=\bigcup_{k=1}^m\widehat U_k$, where $\widehat U_1,...\widehat U_m$ are pairwise disjoint open neighborhoods of $\hat x_1,...,\hat x_m$ in $\widehat X$.

Let $\pi:\widetilde X\to\widehat X$ be a resolution of singularities of $\widehat X$ and then $\xi\circ\pi:\widetilde X\to X$ is a desingularization of $X$. Note that for each point $x\in X$ and any sufficiently small neighborhood $U$ of $x$ in $X$, we have
\begin{equation*}
\begin{split}
\left(R^{n-1}(\xi\circ\pi)_*\mathcal{O}_{\widetilde X}\right)(U)&=H^{n-1}\left(\pi^{-1}\big(\xi^{-1}(U)\big),\mathcal{O}_{\widetilde X}\right)\\
&=\bigoplus_{k=1}^mH^{n-1}\left(\pi^{-1}(\widehat U_k),\mathcal{O}_{\widetilde X}\right)\\
&=\bigoplus_{k=1}^m\left(R^{n-1}\pi_*\mathcal{O}_{\widetilde X}\right)(\widehat U_k).
\end{split}
\end{equation*}
Thus, it follows from $(3)$ of Remark \ref{rational} that $x\in X$ is a weakly rational singularity if and only if all of points $\hat x_1,...,\hat x_m$ are weakly rational singularities of $\widehat X$. In further, if $X$ has at most weakly rational singularities, by the definition, we can obtain
$$\omega_X=\omega_X^{\text{GR}}=(\xi\circ\pi)_*\omega_{\widetilde X}=\xi_*(\pi_*\omega_{\widetilde X})=\xi_*\omega_{\widehat X}.$$
\end{proof}

For our proof of Theorem \ref{SkodaDivision}, we need the following relative version of Grauert-Riemenschneider vanishing theorem for the higher direct images as well.

\begin{theorem} \emph{(\cite{Matsu_morphism}, Corollary 1.5).} \label{G-R_vanishing}
Let $\pi:X\to Y$ be a surjective proper (locally) K\"ahler morphism from a complex manifold $X$ to a complex space $Y$, and $(L,e^{-2\varphi_L})$ be a (possibly singular) Hermitian line bundle on $X$ with semi-positive curvature. Then, the higher direct image sheaf $$R^j\pi_*\Big(\omega_X\otimes\mathcal{O}_X(L)\otimes\mathscr{I}(\varphi_L)\Big)=0,$$
for every $j>\dim X-\dim Y$.
\end{theorem}

For the sake of convenience, we state the following characterization on the slope of plurisubharmonic functions (cf. Corollary 9.3 in \cite{BBEGZ}):

\begin{lemma} \label{slope}
Let $X$ be a locally irreducible complex space and $\varphi\in\emph{Psh}(X)$ with the slope $\nu_x(\varphi)=0$ for some point $x\in X$. Then, for any resolution of singularities $\pi:\widetilde X\to X$ of $X$, $\nu_{\tilde x}(\varphi\circ\pi)=0$ for each $\tilde x\in\pi^{-1}(x)$.
\end{lemma}

Here, the slope of $\varphi$ at $x$ is defined by
$$\nu_x(\varphi):=\sup\left\{\gamma\geq0\ \big|\ \varphi\leq\gamma\log\sum_k|f_k|+O(1)\right\}\in[0,+\infty),$$
where $(f_k)$ are local generators of the maximal ideal $\frak{m}_x$ of $\mathcal{O}_{X,x}$.

\begin{remark} \label{Skoda}
If $X$ is a complex manifold, the slope $\nu_x(\varphi)$ is exactly the usual Lelong number of $\varphi$ at $x$ and the inequality $\nu_x(\varphi)<1$ implies the integrability of $e^{-2\varphi}$ near $x$ by a result of Skoda (see Lemma 5.6 in \cite{De10}).
\end{remark}

Next, we briefly review some relevant algebraic notions and theorems; one can refer to \cite{SH06}.

\begin{definition}
Let $R$ be a commutative ring and $I$ an ideal of $R$. An element $h\in R$ is said to be \emph{integrally dependent} on $I$ if it satisfies a relation
\[ h^d+a_1h^{d-1}+\cdots+a_d=0 \quad (a_k\in{I}^k, 1\le k\le d). \]

The set $\overline{I}$ consisting of all elements in $R$ which are integrally dependent on $I$ is called the \emph{integral closure} of $I$ in $R$. $I$ is called \emph{integrally closed} if $I=\overline{I}$. One can prove that $\overline{I}$ is an ideal of $R$, which is the smallest integrally closed ideal in $R$ containing $I$.
\end{definition}

\begin{definition}
Let $R$ be a commutative ring with identity and let $J\subset I$ be ideals in $R$. $J$ is said to be a \emph{reduction} of $I$ if there exists a nonnegative integer $n$ such that $I^{n+1}=JI^n$.

A reduction $J$ of $I$ is called \emph{minimal} if no ideal strictly contained in $J$ is a reduction of $I$. An ideal that has no reduction other than itself is called \emph{basic}.

One can prove that minimal reductions do exist in Noetherian local rings and an ideal which is a minimal reduction of a given ideal is necessarily basic. Moreover, if $R$ is a Noetherian ring, $J\subset I$ is a reduction of $I$ if and only if $\overline{J}=\overline{I}$.
\end{definition}

\begin{theorem} \emph{(\cite{SH06}, Proposition 8.3.7 \& Corollary 8.3.9).} \label{reduction}
Let $(R,\mathfrak{m})$ be a Noetherian local ring with infinite residue field and $I\subset R$ an ideal. Then every minimal reduction of $I$ can be generated by exactly $l$ (the analytic spread of $I$) elements with $l\leq \dim R$.
\end{theorem}

In the analytic setting, we have the following characterization of integral closure of ideals.

\begin{theorem} \label{LJ-T_IC}\emph{(cf. \cite{LJ-T}, Th\'eor\`eme 2.1).}
Let $X$ be a complex space, $Y$ a proper closed complex subspace (may be non-reduced) of $X$ defined by a coherent $\mathcal{O}_X$-ideal $\mathscr{I}$ and $x\in Y$ a point. Let $\mathscr{J}$ be a coherent $\mathcal{O}_X$-ideal in $\mathcal{O}_X$, $\mathcal{I}$ (resp. $\mathcal{J}$) the germ of $\mathscr{I}$ (resp. $\mathscr{J}$) at $x$. Then the following conditions are equivalent:
\begin{enumerate}
  \item $\mathcal{J}\subset\overline{\mathcal{I}}$;
  \item For every morphism $\pi:X'\to X$ satisfying: (a) $\pi$ is a proper and surjective, (b) $X'$ is a normal complex space, (c) $\mathscr{I}\cdot\mathcal{O}_{X'}$ is an invertible $\mathcal{O}_X$-module, there exists an open neighborhood $U$ of $x$ in $X$ such that
      $$\mathscr{J}\cdot\mathcal{O}_{X'}|_{\pi^{-1}(U)}\subset\mathscr{I}\cdot\mathcal{O}_{X'}|_{\pi^{-1}(U)};$$
  \item If $V$ is an open neighborhood of $x$ on which $\mathscr{I}$ and $\mathscr{J}$ are generated by their global sections. For every system of generators $g_1,...,g_r\in\Gamma(V,\mathscr{I})$ and every $f\in\Gamma(V, \mathscr{J})$, one can find an open neighborhood $V'$ of $x$ and a constant $C>0$ such that
      $$|f(y)|\le C\cdot \sup_{k}|g_k(y)|,\quad \forall y\in V'.$$
\end{enumerate}
\end{theorem}

Combining with the existence of universal denominator for the sheaf $\widehat{\mathcal{O}}_X$ of weakly holomorphic functions on $X$, we can deduce a characterization of the integral closures of ideals in $\widehat{\mathcal{O}}_X$ as follows.

\begin{corollary} \label{IC_weakly}
Let $X$ be a complex space with $x\in X$ a point, $\mathcal{I}=(h_1,...,h_r)\cdot\widehat{\mathcal{O}}_{X,x}$ an ideal in $\widehat{\mathcal{O}}_{X,x}$ and $\boldsymbol{\mathit{u}}$ a universal denominator for $\widehat{\mathcal{O}}_X$ near $x$.

Then, for any $f\in\widehat{\mathcal{O}}_{X,x}$, it follows that $f\in\overline{\mathcal{I}}$, the integral closure of $\mathcal{I}$ in $\widehat{\mathcal{O}}_{X,x}$, if and only if there is a neighborhood $W$ of $x$ in $X$ and a constant $C>0$ such that
$$|\boldsymbol{\mathit{u}}\cdot f(y)|\le C\cdot\sup_{k}|\boldsymbol{\mathit{u}}\cdot h_k(y)|,\quad \forall y\in W.$$
\end{corollary}

\begin{remark} \label{IC_weaklyModi}
Let $X$ be a complex space of pure dimension and $\mathscr{I}\subset\mathcal{O}_X$ a coherent ideal sheaf. Let $\pi:\widetilde X\to X$ be any proper modification from a normal complex space $\widetilde X$ onto $X$ such that $\mathscr{I}\cdot\mathcal{O}_{\widetilde X}=\mathcal{O}_{\widetilde X}(-D)$ for some effective Cartier divisor $D$ on $\widetilde X$. Then, we have $\pi_*\mathcal{O}_{\widetilde X}(-D)=\overline{\mathscr{I}}$, the integral closure of $\mathscr{I}$ in $\widehat{\mathcal{O}}_X$.
\end{remark}

\section{Grauert-Riemenschneider multiplier ideal sheaves on\\ complex spaces with singularities}

In this section, we firstly recall the concept of multiplier ideal sheaves associated to plurisubharmonic functions on complex spaces with singularities (\emph{not} necessarily normal or $\mathbb{Q}$-Gorenstein), which will play a crucial role in this paper. Motivated by the definition of Grauert-Riemenschneider canonical sheaf, we construct the multiplier ideals on complex spaces of pure dimension as follows.

\begin{definition} (\cite{Li_BS}, Definition 3.1). \label{MIS_GR}
Let $X$ be a complex space of pure dimension $n$ and $\varphi\in L_\text{loc}^1(X_\text{reg})$ with respect to the Lebesgue measure. The \emph{Grauert-Riemenschneider multiplier ideal sheaf} associated to the weight $\varphi$ on $X$ is defined to be the $\mathcal{O}_X$-submodule $\mathscr{I}_\text{GR}(\varphi)\subset\mathscr{M}_{X}$ of germs of meromorphic functions $f\in\mathscr{M}_{X,x}$ such that for any holomorphic section $\sigma$ of $\omega_X$ on a neighborhood of $x$, we have $$\Big(\sqrt{-1}\Big)^{n^2}|f|^2e^{-2\varphi}\sigma\wedge\overline\sigma$$ is locally integrable at $x$. In addition, we put $\mathscr{I}_\text{GR}(\varphi)=0$ if $\varphi\equiv-\infty$ on $X$.
\end{definition}

\begin{remark} (Local description, cf. \cite{Li_BS}).
Let $X$ be a (closed) complex subspace of some domain $\Omega$ in $\mathbb{C}^{N}$ and $Z\supset X$ is an $n$-dimensional (reduced) complete intersection (shrinking $\Omega$ if necessary). In particular, the dualizing sheaf $\omega_Z$ is free on $Z$ and so is $\omega_Z|_X$ on $X$. Moreover, it follows that the image of the inclusion $\omega_X\hookrightarrow\omega_Z|_X$ is given by
$$\text{Im}\left(\omega_X\hookrightarrow\omega_Z|_X\right)=\frak{d}_{X,Z}\otimes\omega_Z|_X;$$
in further, we can obtain $\frak{d}_{X,Z}=\big((\mathscr{I}_Z:\mathscr{I}_X)+\mathscr{I}_X\big)/\mathscr{I}_X$ (see \cite{deFD14} for more details).

Let $\pi:\widetilde X\to X$ be a resolution of singularities of the blow-up of $X$ along the ideal sheaf $\frak{d}_{X,Z}$ such that $\frak{d}_{X,Z}\cdot\mathcal{O}_{\widetilde X}=\mathcal{O}_{\widetilde X}(-D_Z)$ for some effective divisor $D_Z$ on $\widetilde X$. Then, for any $f\in\mathscr{I}_\text{GR}(\varphi)_x$ defined on a small enough neighborhood $U$ of $x$ in $X$ and each holomorphic section $\sigma$ of $\omega_X$ on a neighborhood of $x$, we have
$$\int_U\left(\sqrt{-1}\right)^{n^2}|f|^2e^{-2\varphi}\sigma\wedge\overline\sigma
=\int_{\pi^{-1}(U)}\left(\sqrt{-1}\right)^{n^2}|f\circ\pi|^2e^{-2\varphi\circ\pi}\pi^*\sigma\wedge\overline{\pi^*\sigma}<+\infty,$$
which implies that
$$\mathscr{I}_\text{GR}(\varphi)=\pi_*\left(\omega_{\widetilde X}\otimes\pi^*(\omega_Z^{-1}|_X)\otimes\mathcal{O}_{\widetilde X}(D_Z)\otimes\mathscr{I}(\varphi\circ\pi)\right). \eqno{(\spadesuit)}$$
\end{remark}

\begin{remark}
If $X$ is a $\mathbb{Q}$-Gorenstein complex space, by a similar argument on the discrepancies as Proposition 3.4 in \cite{deFD14} (see also \cite{EI15}, Proposition 2.21), it follows that
$$\mathscr{I}_\text{NO}(\varphi)\subset\mathscr{I}(\varphi)\subset\mathscr{I}_\text{GR}(\varphi),$$
where $\mathscr{I}_\text{NO}(\varphi)$ is the Nadel-Ohsawa multiplier ideal sheaf on $X$ introduced in \cite{Li_multiplier} and both of inclusions may be strict; see also \cite{Shibata17} for an algebraic counterpart.
\end{remark}

\begin{theorem} \emph{(=Theorem \ref{GR_weakly}).}
With the same hypotheses as above, then we obtain that $\mathscr{I}_\emph{GR}(\varphi)\subset\widehat{\mathcal{O}}_X$, and $\mathscr{I}_\emph{GR}(\varphi)$ is integrally closed in $\widehat{\mathcal{O}}_X$.
\end{theorem}

\begin{proof}
The integral closedness of $\mathscr{I}_\text{GR}(\varphi)$ in $\widehat{\mathcal{O}}_X$ is straightforward by Corollary \ref{IC_weakly}. Next, we only present the inclusion $\mathscr{I}_\text{GR}(\varphi)_x\subset\widehat{\mathcal{O}}_{X,x}$ for each point $x\in X$. Without loss of generality, we may assume $\varphi=0$. Let $\pi:\widetilde X\to X$ be a resolution of singularities of $X$ with exceptional locus $\text{Ex}(\pi)$ and $f\in\mathscr{I}_\text{GR}(\varphi)_x\subset\mathscr{M}_{X,x}$.\\

Suppose that $f\notin\widehat{\mathcal{O}}_{X,x}$, then $f\circ\pi$ is a meromorphic function with nontrivial poles contained in $\text{Ex}(\pi)$ on $\widetilde X$ after shrinking $X$ and denote by $E$ an irreducible component of the poles of $f\circ\pi$. Let $\sigma_1,...,\sigma_r$ be a system of generators of $\omega_{X,x}$, and then the local integrability of $|f|^2\sigma_k\wedge\overline{\sigma_k}\ (1\leq k\leq r)$ near $x$ implies that $\pi^*\sigma_k$ is a meromorphic $n$-form  on $\widetilde X$, which vanishes on $E$. Take $\pi^*\sigma_{k_0}$ to be the one which has the least vanishing order along $E$ among all $\pi^*\sigma_k$. Thus, the vanishing order $\text{ord}_E\pi^*(f\cdot\sigma_{k_0)}$ of $\pi^*(f\cdot\sigma_{k_0})$ along $E$ is at most $\text{ord}_E\pi^*\sigma_{k_0}-1$.

On the other hand, note that $f\cdot\sigma_{k_0}\in\omega_{X,x}^{\text{GR}}\subset\omega_{X,x}$, i.e., $f\cdot\sigma_{k_0}$ can be generated by $\sigma_1,...,\sigma_r$ near $x$, and then it implies that $$\text{ord}_E\pi^*(f\cdot\sigma_{k_0)}\geq\text{ord}_E\pi^*\sigma_{k_0},$$
which is a contradiction; the proof is concluded.
\end{proof}

\begin{remark}
For the case that $\text{codim}\,X_\text{sing}\geq2$, we can directly obtain the above result by the Riemann extension theorem on locally pure dimensional complex spaces (cf. \cite{GR84}).
\end{remark}

As a consequence, we can deduce the following characterization of weakly rational singularities by the Grauert-Riemenschneider multiplier ideals.

\begin{proposition}
Let $X$ be a locally irreducible complex space of dimension $n$ and $\varphi\in\emph{Psh}(X)$ with the slope $\nu_x(\varphi)=0$ for every point $x\in X$. Then, $X$ has at most weakly rational singularities if and only if $\mathscr{I}_\emph{GR}(\varphi)=\widehat{\mathcal{O}}_X$.
\end{proposition}

\begin{proof}
Let $\pi:\widetilde X\to X$ be a resolution of singularities. If $X$ has at most weakly rational singularities, then for any holomorphic section $\sigma$ of $\omega_X$, it follows that $\pi^*\sigma$ is a holomorphic $n$-form on $\widetilde X$. Thus, the necessity is a direct application of Lemma \ref{slope} and Remark \ref{Skoda}. The sufficiency is straightforward since $\mathcal{O}_X\subset\widehat{\mathcal{O}}_X$ and $\varphi$ is upper semi-continuous on $X$.
\end{proof}

\begin{remark}
$(1)$ If $\varphi$ is locally bounded, the statement is also valid for any complex space of pure dimension.

$(2)$ In particular, if $X$ is a normal complex space, then $X$ has at most weakly rational singularities if and only if $\mathscr{I}_\text{GR}(\varphi)=\mathcal{O}_X$, which implies that on normal complex spaces with weakly rational singularities, the plurisubharmonic weights with zero slope do not add the singularities in the sense of Grauert-Riemenschneider multiplier ideals.
\end{remark}

Let $X|_U\hookrightarrow\Omega\subset\mathbb{C}^{N}$ be any local embedding of $X$ for some open set $U\subset X$ and $Z\supset X|_U$ an $n$-dimensional (reduced) complete intersection. Note that $$\text{Im}\left(\omega_X\hookrightarrow\omega_Z|_X\right)=\frak{d}_{X,Z}\otimes\omega_Z|_X,$$
and then we can deduce that $\omega_Z|_X\otimes\mathscr{I}_\text{GR}(\varphi+\log|\frak{d}_{X,Z}|)$ globally defines a coherent analytic sheaf on $X$ for any $\varphi\in\text{Psh}(X)$. Thus, relying on the multiplier ideals defined above, we will derive a form of Nadel vanishing theorem on complex spaces with singularities as follows.

\begin{proposition} \label{Nadel_vanishing}
Let $(X,\omega)$ be a weakly pseudoconvex K\"ahler space of pure dimension $n$. Let $(L,e^{-2\varphi_L})$ be a holomorphic line bundle on $X$ equipped with a singular Hermitian metric of plurisubharmonic weight $\varphi_L$ and $\sqrt{-1}\partial\overline\partial\varphi_L\geq\varepsilon\omega$ for some positive continuous function $\varepsilon$ on $X_\text{\emph{reg}}$. Then,
$$H^q\big(X,\omega_Z|_X\otimes\mathscr{I}_\emph{GR}(\varphi_L+\log|\frak{d}_{X,Z}|)\otimes\mathcal{O}_X(L)\big)=0$$
for all $q\geq1$, where $X$ is locally contained in an $n$-dimensional complete intersection $Z$ and $\omega_X=\frak{d}_{X,Z}\otimes\omega_Z|_X$.
\end{proposition}

\begin{proof}
Thanks to Theorem \ref{NaDe_vanishing}, it is sufficient to prove
$$\omega_X^{\text{GR}}(\varphi_L)=\omega_Z|_X\otimes\mathscr{I}_\text{GR}(\varphi_L+\log|\frak{d}_{X,Z}|).$$

Let $\pi:\widetilde X\to X$ be a resolution of singularities of the blow-up of $X$ along the ideal sheaf $\frak{d}_{X,Z}$, and $\frak{d}_{X,Z}\cdot\mathcal{O}_{\widetilde X}=\mathcal{O}_{\widetilde X}(-D_Z)$ for some effective divisor $D_Z$ (not necessarily SNC) on $\widetilde X$. Then, it follows from the projection formula and $(\spadesuit)$ that
\begin{equation*}
\begin{split}
\omega_X^{\text{GR}}(\varphi_L)&=\pi_*\left(\omega_{\widetilde X}\otimes\mathscr{I}(\varphi_L\circ\pi)\right)\\
&=\pi_*\left(\pi^*(\omega_Z|_X)\otimes\mathcal{O}_{\widetilde X}(-D_Z)\otimes\omega_{\widetilde X}\otimes\pi^*(\omega_Z^{-1}|_X)\otimes\mathcal{O}_{\widetilde X}(D_Z)\otimes\mathscr{I}(\varphi_L\circ\pi)\right)\\
&=\pi_*\left(\pi^*(\omega_Z|_X)\otimes\omega_{\widetilde X}\otimes\pi^*(\omega_Z^{-1}|_X)\otimes\mathcal{O}_{\widetilde X}(D_Z)\otimes\mathscr{I}(\varphi_L\circ\pi+\log|\frak{d}_{X,Z}|\circ\pi)\right)\\
&=\omega_Z|_X\otimes\pi_*\left(\omega_{\widetilde X}\otimes\pi^*(\omega_Z^{-1}|_X)\otimes\mathcal{O}_{\widetilde X}(D_Z)\otimes\mathscr{I}(\varphi_L\circ\pi+\log|\frak{d}_{X,Z}|\circ\pi)\right)\\
&=\omega_Z|_X\otimes\mathscr{I}_\text{GR}(\varphi_L+\log|\frak{d}_{X,Z}|).
\end{split}
\end{equation*}
\end{proof}

Note when $X$ is 1-Gorenstein, one can infer that $$\omega_Z|_X\otimes\mathscr{I}_\text{GR}(\varphi_L+\log|\frak{d}_{X,Z}|)=\omega_X\otimes\mathscr{I}_\text{GR}(\varphi_L).$$
Thus, the above result immediately implies that

\begin{corollary} \label{Nadel_vanishing1G}
With the same hypotheses as above, we in addition assume that $X$ is a 1-Gorenstein complex space. Then, we have
$$H^q\big(X,\omega_X\otimes\mathcal{O}_X(L)\otimes\mathscr{I}_\emph{GR}(\varphi_L)\big)=0$$
for all $q\geq1$.
\end{corollary}

As a result, by combining with Richberg's result \cite{Richberg68}, we will deduce a solution to the Levi problem on $1$-Gorenstein complex spaces with weakly rational singularities.

\begin{corollary}
Let $X$ be a $1$-Gorenstein complex space of dimension $n$ with weakly rational singularities. Then, the following statements are equivalent:

$(1)$ There is a continuous strictly plurisubharmonic exhaustion function on $X$.

$(2)$ $X$ is a Stein space.
\end{corollary}

\begin{proof}
"$(2)\Rightarrow(1)$". Since $X$ is a Stein space, it carries a real analytic strictly plurisubharmonic exhaustion function (cf. \cite{Nara_Levi}).\\

"$(1)\Rightarrow(2)$". By Richberg's result, there exists a smooth strictly plurisubharmonic exhaustion function $\psi$ on $X$, which implies $X$ is a K\"ahler space. Moreover, since $X$ is $1$-Gorenstein, we have $\mathscr{I}_\text{GR}(\varphi)\subset\mathcal{O}_X$ for any $\varphi\in\text{Psh}(X)$.

Let $(x_k)\subset X$ be any discrete sequence of points and $u$ a quasi-plurisubharmonic function on $X$ such that $$u\sim c_k\log(|f_{k1}|+\cdots+|f_{kl_k}|)$$ near each $x_k$, where $(f_{k1},...,f_{kl_k})$ are generators of the maximal ideal $\mathfrak{m}_{x_k}$ and $c_k\gg0$ such that $$\text{Supp}(\mathcal{O}_X/\mathscr{I}_\text{GR}(u))=\{x_k\}$$ (in fact, $c_k\geq n$ is enough); here, the weak rationality of singularities of $X$ is necessary. Take a holomorphic line bundle $L$ on $X$ with $\mathcal{O}_X(L^{-1})=\omega_X$ and a singular Hermitian metric on $L$ of the form $h_L=h_0e^{-2(u+\chi\circ\psi)}$, where $h_0$ is a smooth metric on $L$ and $\chi$ is a convex increasing function of arbitrary fast growth at infinity such that $$\sqrt{-1}\Theta_{L,h_L}=\sqrt{-1}\left(\Theta_{L,h_0}+2\partial\overline\partial(u+\chi\circ\psi)\right)\geq\varepsilon\omega$$
for some positive continuous function $\varepsilon$ on $X_\text{reg}$. Then, it follows that $\mathscr{I}_\text{GR}(h_L)=\mathscr{I}_\text{GR}(u)$.

Consider the long exact sequence of cohomology associated to the short exact sequence $$0\rightarrow\mathscr{I}_\text{GR}(h_L)\rightarrow\mathcal{O}_X\rightarrow\mathcal{O}_X/\mathscr{I}_\text{GR}(h_L)\rightarrow0$$
twisted by $\omega_X\otimes\mathcal{O}_X(L)$. Applying Corollary \ref{Nadel_vanishing1G}, we obtain the surjectivity
$$H^0\big(X,\mathcal{O}_X\big)\twoheadrightarrow H^0\big(X,\mathcal{O}_X/\mathscr{I}_\text{GR}(u)\big)$$
by our choice of $L$, which implies that $X$ is holomorphically separable and convex, i.e., $X$ is a Stein space.
\end{proof}

\section{Proofs of main results}

\subsection{Proof of Theorem \ref{SkodaDivision}}

Firstly, we state a result on the existence of reductions for coherent (fractional) ideal sheaves on complex spaces.

\begin{lemma} \label{Exis_reduction}
Let $X$ be a complex space of dimension $n$ and $\mathfrak{a}\subset\widehat{\mathcal{O}}_X$ a non-zero ideal. Then
there exists an open covering $\{U_\alpha\}_{\alpha\in\mathbb{N}}$ of $X$ such that $\mathfrak{a}|_{U_\alpha}$ has a reduction $\mathfrak{b}_\alpha$ generated by at most $n$ elements.
\end{lemma}

Thanks to Corollary \ref{IC_weakly}, the proof is similar to the argument of Example 9.6.19 in \cite{La04} by considering the normalized blow-up of the normalization of $X$, where an analytic Bertini-type theorem is necessary.\\

\noindent{\textbf{Proof of Theorem \ref{SkodaDivision}.}} By the definition of Grauert-Riemenschneider multiplier ideals, we can deduce the inclusion
$$\mathfrak{a}\cdot\mathscr{I}_\text{GR}\big(\varphi+(k-1)\varphi_\mathfrak{a}\big)\subset\mathscr{I}_\text{GR}\big(\varphi+k\varphi_\mathfrak{a}\big).$$
Thus, in order to prove Theorem \ref{SkodaDivision}, it is sufficient to show the reverse inclusion.

As the question is local, without loss of generality, we may assume that $X$ is a (closed) complex subspace of some domain $\Omega$ in $\mathbb{C}^{N}$ and $Z\supset X$ is an $n$-dimensional (reduced) complete intersection.\\

\textbf{Case (i).} When $r\leq n$, i.e., $q=r$.

Let $g=(g_1,\dots,g_r)$ be a system of generators of $\mathfrak{a}$. Let $\pi:\widetilde X\to\widehat X\stackrel{\xi}{\to}X$ be a log resolution of $\frak{d}_{X,Z}$ and $\frak{a}$ such that $\frak{d}_{X,Z}\cdot\mathcal{O}_{\widetilde X}=\mathcal{O}_{\widetilde X}(-D_Z)$ and $\frak{a}\cdot\mathcal{O}_{\widetilde X}=\mathcal{O}_{\widetilde X}(-F)$ for some effective divisors $D_Z$ and $F$ on $\widetilde X$, where $\xi:\widehat X\to X$ is the normalization of $X$. Denote by
\begin{equation*}
\begin{split}
\mathscr{A}_m:=&\ \omega_{\widetilde X}\otimes\pi^*(\omega_Z^{-1}|_X)\otimes\mathcal{O}_{\widetilde X}(D_Z)\otimes\mathscr{I}(\varphi\circ\pi+m\varphi_\mathfrak{a}\circ\pi)\\
=&\ \omega_{\widetilde X}\otimes\pi^*(\omega_Z^{-1}|_X)\otimes\mathcal{O}_{\widetilde X}(D_Z)\otimes\mathscr{I}(\varphi\circ\pi)\otimes\mathcal{O}_{\widetilde X}(-mF)
\end{split}
\end{equation*}
for any $m\in\mathbb{N}$, and consider the Koszul complex determined by $g_1,\dots,g_r$:
$$0\to\Lambda^rV\otimes\mathcal{O}_{\widetilde X}(rF)\to\cdots\to\Lambda^2V\otimes\mathcal{O}_{\widetilde X}(2F)\to V\otimes\mathcal{O}_{\widetilde X}(F)\to\mathcal{O}_{\widetilde X}\to0,$$
where $V$ is the vector space spanned by $g_1,\dots,g_r$. Note that the Koszul complex is locally split and its syzygies are locally free, so twisting through by any coherent sheaf will preserve the exactness. Then, by twisting with $\mathscr{A}_k\ (k\geq r=q)$, we obtain the following long exact sequence
$$0\to\Lambda^rV\otimes\mathscr{A}_{k-r}\to\cdots\to\Lambda^2V\otimes\mathscr{A}_{k-2}\to V\otimes\mathscr{A}_{k-1}\to\mathscr{A}_k\to0.\eqno{(\star)}$$

On the other hand, for any $m\in\mathbb{N}$, we can derive the vanishing of higher direct images $R^j\pi_*\mathscr{A}_m=0\ (\forall j>0)$ by Theorem \ref{G-R_vanishing}. Note that
$$\mathscr{I}_\text{GR}\big(\varphi+m\varphi_\mathfrak{a}\big)=\pi_*\mathscr{A}_m$$
by $(\spadesuit)$, and then by taking direct
images of $(\star)$ we will deduce the following so-called exact Skoda complex (cf. \cite{La04}, p. 228):
$$0\to\Lambda^rV\otimes\mathscr{I}_\text{GR}\big(\varphi+(k-r)\varphi_\mathfrak{a}\big)\to\cdots\to V\otimes\mathscr{I}_\text{GR}\big(\varphi+(k-1)\varphi_\mathfrak{a}\big)\to\mathscr{I}_\text{GR}\big(\varphi+k\varphi_\mathfrak{a}\big)\to0.$$
In particular, the map $$V\otimes\mathscr{I}_\text{GR}\big(\varphi+(k-1)\varphi_\mathfrak{a}\big)\to\mathscr{I}_\text{GR}\big(\varphi+k\varphi_\mathfrak{a}\big)$$
is surjective, by which we can infer that $$\mathscr{I}_\text{GR}\big(\varphi+k\varphi_\mathfrak{a}\big)\subset\frak{a}\cdot\mathscr{I}_\text{GR}\big(\varphi+(k-1)\varphi_\mathfrak{a}\big).$$

\textbf{Case (ii).} When $r>n$, i.e., $q=n$.

By Lemma \ref{Exis_reduction}, we may assume that $\mathfrak{b}$ is a reduction of $\mathfrak{a}$ generated by $n$ elements $\widetilde g_1,...,\widetilde g_n$. Consider a log resolution $\pi:\widetilde X\to\widehat X\stackrel{\xi}{\to}X$ of $\frak{d}_{X,Z},\ \frak{a}$ and $\frak{b}$ such that $\frak{d}_{X,Z}\cdot\mathcal{O}_{\widetilde X}=\mathcal{O}_{\widetilde X}(-D_Z)$ and $\frak{a}\cdot\mathcal{O}_{\widetilde X}=\frak{b}\cdot\mathcal{O}_{\widetilde X}=\mathcal{O}_{\widetilde X}(-F)$ for some effective divisors $D_Z$ and $F$ on $\widetilde X$. Then, by the same argument as above, we can deduce the following exact Skoda complex:
$$0\to\Lambda^nV\otimes\mathscr{I}_\text{GR}\big(\varphi+(k-n)\varphi_\mathfrak{a}\big)\to\cdots\to V\otimes\mathscr{I}_\text{GR}\big(\varphi+(k-1)\varphi_\mathfrak{a}\big)\to\mathscr{I}_\text{GR}\big(\varphi+k\varphi_\mathfrak{a}\big)\to0.$$
for any $k\geq n=q$, where $V$ is the vector space spanned by $\widetilde g_1,...,\widetilde g_n$. Therefore, it follows that $$\mathscr{I}_\text{GR}\big(\varphi+k\varphi_\mathfrak{a}\big)\subset\frak{b}\cdot\mathscr{I}_\text{GR}\big(\varphi+(k-1)\varphi_\mathfrak{a}\big)\subset
\frak{a}\cdot\mathscr{I}_\text{GR}\big(\varphi+(k-1)\varphi_\mathfrak{a}\big);$$
the proof is concluded.
\hfill $\Box$

\subsection{Proof of Theorem \ref{BS_weakly}}

Firstly, by Theorem \ref{reduction}, it is sufficient to prove the desired result for the case $r\leq n$, i.e., $q=r$. In fact, we then can replace $\mathcal{I}$ by a minimal reduction $\mathcal{J}$ of $\mathcal{I}$, generated by at most $n$ elements, and in further obtain $$\overline{\mathcal{I}^{k+q-1}}=\overline{\mathcal{J}^{k+q-1}}\subset\mathcal{J}^{k}\subset\mathcal{I}^{k}.$$

Let $g=(g_1,\dots,g_r)$ be a system of generators of $\mathcal{I}\subset\widehat{\mathcal{O}}_{X,x}$. As the statement is local, we may assume that all functions $g_1,...,g_r\in\widehat{\mathcal{O}}(X)$ and put $\frak{a}=(g_1,...,g_r)\cdot\widehat{\mathcal{O}}_X$. For any $f\in\overline{\mathcal{I}^{k+q-1}}$, it follows from Corollary \ref{IC_weakly} that $|f|\leq C\cdot|g|^{k+q-1}$ for some constant $C\geq0$ near $x$ on $X$, and then by weak rationality of the point $x$, we deduce that $f\in\mathscr{I}_\text{GR}(\frak{a}^{k+q-1})_x$. Therefore, in order to prove Theorem \ref{BS_weakly}, by Theorem \ref{GR_weakly} it is sufficient to show that $$\mathscr{I}_\text{GR}(\frak{a}^{p})=\frak{a}\cdot\mathscr{I}_\text{GR}(\frak{a}^{p-1})$$
for any integer $p\geq r$ on $X$, which is an immediate consequence of Theorem \ref{SkodaDivision}.
\hfill $\Box$

\subsection{Proof of Theorem \ref{BS_1D}}

Relying on the reduction argument as above, we may consider a principal ideal $\mathcal{J}=(g)\cdot\widehat{\mathcal{O}}_{X,x}\subset\widehat{\mathcal{O}}_{X,x}$ and $f\in\overline{\mathcal{J}}$. Then, it follows from Corollary \ref{IC_weakly} that $f/g\in\widehat{\mathcal{O}}_{X,x}$, which implies that $\overline{\mathcal{J}}\subset\mathcal{J}$, i.e., the original Brian\c{c}on-Skoda theorem holds for the Noetherian ring $\widehat{\mathcal{O}}_{X,x}$.\\

Note that each reduced complex space of dimension one is Cohen-Macaulay, and then it follows that $x\in X$ is a weakly rational singularity if and only if $x\in X$ is a rational singularity if and only if $x\in X$ is a regular point.
\medskip

``$(2)\Longrightarrow(3)$''. If $x\in X$ is a regular point, then $\mathcal{O}_{X,x}$ is a principal ideal domain. Then, by Proposition 1.5.2 in \cite{SH06} (or the original result of Brian\c{c}on and Skoda), it follows that $\overline{\mathcal{I}^k}=\mathcal{I}^k$ for any $k\in\mathbb{N}$.
\medskip

``$(3)\Longrightarrow(2)$''. Suppose that $x\in X$ is a singularity. As the question is local, we may assume that $X$ is a (closed) complex subspace of some domain $\Omega$ in $\mathbb{C}^{N}$, where $N$ is the embedding dimension of $X$ at $x$ (= the origin $o$ of $\mathbb{C}^{N}$). In particular, we have
$$\text{ord}\mathscr{I}_{X,o}:=\min\{\,\text{ord}_{o}(f)\,|\,f\in\mathscr{I}_{X,o}\}\geq2;$$
otherwise, the embedding dimension of $X$ at $o$ is at most $N-1$.

By the Weierstrass Preparation Theorem, in some local coordinates $(z';z'')=(z_1;z_{2},...,z_N)$ near $o$, there exist Weierstrass polynomials
$$P_k(z)=z_k^{m_k}+a_{1k}z_k^{m_k-1}+\cdots+a_{m_kk}\in\mathcal{O}_{k-1}[z_k]\cap\mathscr{I}_{X,o},\ k=2,...,N, \eqno{(*)}$$
with $m_k=\mbox{ord}_oP_k$. Hence, we have
$$a_{jk}(z_1,...,z_{k-1})\in\mathfrak{m}_{k-1}^j,\ 2\leq k\leq N,\ 1\leq j\leq m_k. \eqno{(**)}$$

Denote by $\mathcal{I}$ the ideal in $\mathcal{O}_{X,o}$ generated by the germ of holomorphic function $\widehat{z}_1\in\mathcal{O}_{X,o}$, where $\widehat{z}_k$ are the residue classes of $z_k$ in $\mathcal{O}_{X,o}$. Then, combining $(*)$ with $(**)$, we obtain that the integral closure $\overline{\mathcal{I}}$ of $\mathcal{I}$ in $\mathcal{O}_{X,o}$ is $\mathfrak{m}_{X,o}=(\widehat{z}_1,...,\widehat{z}_{N})\cdot\mathcal{O}_{X,o}$, the maximal ideal of $\mathcal{O}_{X,o}$. Moreover, since ord$\mathscr{I}_{X,o}\geq2$, we have $\widehat{z}_k\not\in\mathcal{I},\ 2\leq k\leq n$. Then, we have $\overline{\mathcal{I}}\not\subset\mathcal{I}$, which contradicts to validity of the original Brian\c{c}on-Skoda theorem for $\mathcal{O}_{X,o}$; i.e., $x\in X$ is a regular point.
\hfill $\Box$

\subsection{Proof of Theorem \ref{BS_2D}}

\emph{Necessity.} It is a direct application of Theorem \ref{BS_weakly}.
\smallskip

\emph{Sufficiency.} Firstly, we show that the desired result holds for the case when $x\in X$ is a normal point. As the question is local, by Artin's algebraization theorem, we may assume that $X$ is a normal, affine algebraic variety over $\mathbb{C}$.

Note the Cohen-Macaulayness of normal complex surfaces, and then we can deduce from Theorem 1.1 in \cite{Shibata17} that if the original Brian\c{c}on-Skoda theorem holds for $\mathcal{O}_{X,x}$, essentially of finite type over $\mathbb{C}$, then $x\in X$ is an algebraic rational singularity, which is actually equivalent to the fact that $x\in X$ is an analytic rational singularity (cf. \cite{K-M98}, Corollary 5.11).

In general, we consider the normalization $\xi:\widehat X\to X$ of $X$. By Proposition \ref{invariance_weakly}, it is enough to present that $\hat x\in\widehat X$ is a weakly rational singularity for any $\hat x\in\xi^{-1}(x)$. Let $\hat{\mathcal{I}}\subset\mathcal{O}_{\widehat X,\hat x}$ be a nonzero ideal with $r$ generators $\hat g_1,...,\hat g_r$. Note that $\widehat{\mathcal{O}}_{X,x}\cong\bigoplus_{j=1}^m\mathcal{O}_{\widehat X,\hat x_j}$, where $\hat x_1,...,\hat x_m$ are the distinct points of a fiber $\xi^{-1}(x)$ and we take $\hat x=\hat x_1$. Then, there exist $g_1,...,g_r\in\widehat{\mathcal{O}}_{X,x}$ such that for any $1\leq j\leq m$, $g_j\circ\xi$ equals to $\hat g_j$ near $\hat x_1$ and equals to one near $\hat x_2,...,\hat x_m$. Denote by $\mathcal{I}$ the ideal in $\widehat{\mathcal{O}}_{X,x}$ generated by $g_1,...,g_r$. Then, by the assumption, we have $\overline{\mathcal{I}^{k+q-1}}\subset\mathcal{I}^k$ for any $k\in\mathbb{N}$. Thanks to Corollary \ref{IC_weakly}, it follows that $\overline{\hat{\mathcal{I}}^{k+q-1}}\subset\hat{\mathcal{I}}^k$, which implies that $\hat x\in X$ is a rational singularity by the initial argument; the proof is complete.
\hfill $\Box$

\appendix
   \renewcommand{\appendixname}{Appendix~\Alph{section}}

\section{Some remarks on Skoda's $L^2$ division theorem} \label{A}

\subsection{Skoda's $L^2$ division theorem: a trivial generalization} \label{Non_division}

For the readers' convenience, here we restate some facts on a trivial generalization of Skoda's $L^2$ division theorem (for holomorphic functions) on singular complex spaces in \cite{Li_BS}. The approach in the arguments is the same as in the author's joint paper \cite{G-L18} (see also \cite{Li_BSOT, Ohsawa_book18}).

\begin{theorem} \emph{(\cite{Li_BS}, Theorem 4.1).} \label{Skoda_negative}
Let $\Omega\subset\mathbb{C}^n\ (n\geq2)$ be a domain, $A\subset\Omega$ an analytic set of pure dimension $d$ through the origin $o$. Then, for small enough ball $B_{r}(o)\subset\Omega$, the $L^2$ division theorem holds on $B_{r}(o)\cap A$ if and only if $o$ is a regular point of $A$.
\end{theorem}

\begin{remark}
In addition, via an argument on the integral closure of ideals as in the proof of Theorem 1.3 in \cite{G-L18}, we are able to deduce the following:

Let $\Omega\subset\mathbb{C}^n\ (n\geq2)$ be a domain and $A\subset\Omega$ an analytic set of pure dimension $d$. Then, the $L^2$ division of bounded holomorphic sections will imply that
$$\text{ord}\mathscr{I}_{A,x}:=\min\{\,\mbox{\text{ord}}_x(f)\,|\,f\in\mathscr{I}_{A,x}\,\}\leq d,$$
for every point $x\in A$.
\end{remark}

Firstly we state a basic estimate of the volume of an analytic subset as follows.

\begin{lemma} \emph{(\cite{G-L18}, Lemma 2.3).} \label{finiteness}
Let $\Omega\subset\mathbb{C}^n\ (n\geq2)$ be a domain and $A\subset\Omega$ an analytic set of pure dimension $d$ through the origin $o$. Then, there exists a neighborhood $U$ of $o$ such that, for any $0\leq\varepsilon<1$,
$$\int_{U{\cap}A}\frac{1}{(|z_1|^2+\cdots+|z_n|^2)^{d+\varepsilon-1}}dV_A<+\infty,$$
where $dV_A=\frac{1}{d!}\omega^d|_{A_\emph{reg}}\ \text{and}\ \omega=\frac{\sqrt{-1}}{2}\sum\limits_{k=1}^{n}dz_k{\wedge}d\bar{z}_k$.
\end{lemma}

The above statement is in fact slightly stronger than Lemma 2.3 in \cite{G-L18}, but it is proved in the same way. As a result, with such a minor modification, it turns out that it is not necessary to use the strong openness of multiplier ideal sheaves associated to plurisubharmonic functions in the proof of the main result, Theorem 1.2, in \cite{G-L18}.\\

\noindent{\textbf{\emph{Proof of Theorem} \ref{Skoda_negative}.}} The sufficiency is a straightforward application of Skoda's $L^2$ division theorem (see \cite{De10}, Theorem 11.13) on $d$-dimensional balls centered at the origin in $\mathbb{C}^d$. Thus, we only need to prove the necessity.\\

Suppose that $o$ is a singular point of $A$ with $1\leq\dim_oA=d\leq n-1$. Thanks to the local parametrization theorem for analytic sets, we can find a local coordinate system $$(z';z'')=(z_1,...,z_d;z_{d+1},...,z_n)$$ near $o$ such that for some constant $C>0$, we have $|z''|\leq C\cdot|z'|$ for any $z\in A$ near $o$.

Let $\mathcal{I}\subset\mathcal{O}_{A,o}$ be the ideal generated by holomorphic functions $\widehat{z}_1,...,\widehat{z}_{d}\in\mathcal{O}_{A,o}$, where $\mathcal{O}_A=\mathcal{O}_{\Omega}/\mathscr{I}_A\big|_A$ and $\widehat{z}_k$ are the residue classes of $z_k$ in $\mathcal{O}_{A,o}$. From the non-smoothness of $A$ at the origin $o$, we deduce that the embedding dimension $\dim_{\mathbb{C}}(\mathfrak{m}_{A,o}/\mathfrak{m}^2_{A,o})\geq d+1$ of $A$ at $o$, which implies that there exists $d+1\leq k_0\leq n$ such that $\widehat{z}_{k_0}\not\in\mathcal{I}$.\\

On the other hand, since $|z''|\leq C\cdot|z'|$ for any $z\in A$ near $o$, we infer that
\begin{center}
$|z_{k_0}|^2\leq C^2\cdot|z'|^2$ and $\frac{|z|^2}{1+C^2}\leq|z'|^2$
\end{center}
on $U\cap A$ for some neighborhood $U$ of $o$. Thus, by Lemma \ref{finiteness}, for some smaller neighborhood $U$ of $o$ and any $0<\varepsilon<1$, it follows that
$$\int_{U{\cap}A}\frac{|{z}_{k_0}|^2}{|z'|^{2(d+\varepsilon)}}dV_A\leq C^2(1+C^2)^{d+\varepsilon-1}\cdot\int_{U{\cap}A}|z|^{-2(d+\varepsilon-1)}dV_A<+\infty.$$

Take a small ball $B_r(o)\subset U$ and apply the $L^2$ division theorem with $\varphi=0$ on $B_r(o)\cap A$. Then, there exist holomorphic functions $h_k\in\mathcal{O}(B_r(o)\cap A)$ such that $$\widehat{z}_{k_0}=h_1{\cdot}\widehat{z}_1+\cdots+h_d{\cdot}\widehat{z}_d$$ on $B_r(o)\cap A$, which contradicts that $\widehat{z}_{k_0}$ is not in $\mathcal{I}$. Thus the proof is concluded.
\hfill $\Box$

\subsection{An alternative argument on the proof of Theorem \ref{BS_weakly}}

Analogous to the arguments as in \cite{Li_BS}, we can also give an alternative proof of Theorem \ref{BS_weakly} for $1$-Gorenstein case, based on a version of Skoda's $L^2$ division theorem for holomorphic sections on complete K\"ahler manifolds (see \cite{De10}, Theorem 11.8), i.e.,

\begin{theorem} \label{De-Skoda}
Let $X$ be a complete K\"ahler manifold equipped with a K\"ahler metric $\omega$ on $X$, $\texttt{g}:E\to Q$ a surjective morphism of hermitian vector bundles and $L\to X$ a hermitian line bundle. Set $$r=\emph{rk}\ E,\ q=\emph{rk}\ Q,\ m=\min\{n-k,r-q\}$$ and suppose $E\geq_m0$,
$$\sqrt{-1}\Theta(L)-(m+\varepsilon)\sqrt{-1}\Theta(\det\ Q)\geq0$$
for some $\varepsilon>0$. Then for every $D''$-closed form $f$ of type $(n,k)$ with values in $Q\otimes L$ such that
$$I=\int_X\langle\widetilde{\texttt{gg}^\star}f,f\rangle(\det\texttt{gg}^\star)^{-m-1-\varepsilon}dV_\omega<+\infty,$$
there exists a $D''$-closed form $h$ of type $(n,k)$ with values in $E\otimes L$ such that $f=\texttt{g}\cdot h$ and
$$\int_X|h|^2(\det\texttt{gg}^\star)^{-m-\varepsilon}dV_\omega\leq(1+m/\varepsilon)I.$$
\end{theorem}

\begin{theorem}
Let $X$ be a complex space of pure dimension $n$ with $x\in X$ a singularity, and $\mathcal{I}\subset\widehat{\mathcal{O}}_{X,x}$ a nonzero ideal with $r$ generators. If $x\in X$ is a weakly rational singularity and $X$ is $1$-Gorenstein near $x$, then the original Brian\c{c}on-Skoda theorem holds for the Noetherian ring $\widehat{\mathcal{O}}_{X,x}$.
\end{theorem}

\vspace{.1in} {\em Acknowledgements}. The author would like to sincerely thank Professor Xiangyu Zhou for guiding to the subject on Brian\c{c}on-Skoda theorem, as well as his generous support and encouragements.


\begin{thebibliography}{123}

\bibitem{ASS10} M. Andersson, H. Samuelsson, J. Sznajdman, \emph{On the Brian\c{c}on-Skoda theorem on a singular variety}, Ann. Inst. Fourier (Grenoble) 60 (2010), 417--432.
\bibitem{A-S84} M. Andreatta, A. Silva, \emph{On weakly rational singularities in complex analytic geometry}, Ann. Mat. Pura Appl. (4) 136 (1984), 65--76.
\bibitem{BBEGZ} R. Berman, S. Boucksom, P. Eyssidieux, et al., \emph{K\"ahler-Einstein metrics and the K\"ahler-Ricci flow on log Fano varieties}, J. Reine Angew. Math. 751 (2019), 27--89.
\bibitem{B-S} J. Brian\c{c}on, H. Skoda, \emph{Sur la cl\^{o}ture int\'{e}grale d'un id\'{e}al de germes de fonctions holomorphes en un point de $\mathbb{C}^n$}, C. R. Acad. Sc. Paris, S\'{e}r. A 278 (1974), 949--951.
\bibitem{CM85} M. Col\c{t}oiu, N. Mihalache, \emph{Strongly plurisubharmonic exhaustion functions on 1-convex spaces}, Math. Ann. 270 (1985), 63--68.
\bibitem{deFD14} T. de Fernex, R. Docampo, \emph{Jacobian discrepancies and rational singularities}, J. Eur. Math. Soc. 16 (2014), 165--199.
\bibitem{De82} J.-P. Demailly, \emph{Estimations $L^2$ pour l'op\'erateur $\bar\partial$ d'un fibr\'e vectoriel holomorphe semi-positif au-dessus d'une vari\'et\'e k\"ahl\'erienne compl\`ete}, Ann. Sci. \'Ecole Norm. Sup. (4) 15 (1982), 457--511.
\bibitem{De10} J.-P. Demailly, \emph{Analytic Methods in Algebraic Geometry}, Higher Education Press, Beijing, 2010.
\bibitem{EI15} L. Ein, S. Ishii, \emph{Singularities with respect to Mather-Jacobian discrepancies}, in Commutative Algebra and Noncommutative Algebraic Geometry II, Math. Sci. Res. Inst. Publ., Vol. 68,  Cambridge Univ. Press, New York, 2015, pp. 125--168.
\bibitem{GPR94} H. Grauert, T. Peternell, R. Remmert, \emph{Several Complex Variables VII}, Encyclopaedia of Mathematical Sciences, Vol. 74, Springer-Verlag, Berlin, 1994.
\bibitem{GR84} H. Grauert, R. Remmert, \emph{Coherent Analytic Sheaves}, Grundlehren Math. Wiss., Vol. 265, Springer-Verlag, Berlin, 1984.
\bibitem{G-L18} Q. A. Guan, Z. Q. Li, \emph{A characterization of regular points by Ohsawa-Takegoshi extension theorem}, J. Math. Soc. Japan 70 (2018), 403--408.
\bibitem{G-Z_optimal}  Q. A. Guan, X. Y. Zhou, \emph{A solution of an $L^2$ extension problem with an optimal estimate and applications},  Ann. of Math. (2) 181 (2015), 1139--1208.
\bibitem{G-Z_open}  Q. A. Guan, X. Y. Zhou, \emph{A proof of Demailly's strong openness conjecture}, Ann. of Math. (2) 182 (2015), 605--616. See also arXiv: 1311.3781.
\bibitem{HenPa99} G. Henkin, M. Passare, \emph{Abelian differentials on singular varieties and variations on a theorem of Lie-Griffiths}, Invent. math. 135 (1999), 297--328.
\bibitem{Huneke92} C. Huneke, \emph{Uniform bounds in Noetherian rings}, Invent. Math. 107 (1992), 203--223.
\bibitem{H-V} E. Hyry and O. Villamayor, \emph{A Brian\c{c}on-Skoda theorem for isolated singularities}, J. Algebra 204 (1998), 656--665.
\bibitem{Ishii_sing} S. Ishii, \emph{Introduction to singularities}, Springer, Tokyo, 2014.
\bibitem{K-M98} J. Koll\'ar, S. Mori, \emph{Birational geometry of algebraic varieties}, Cambridge Tracts in Mathematics, Vol. 134, Cambridge Univ. Press, Cambridge, 1998.
\bibitem{La04} R. Lazarsfeld, \emph{Positivity in Algebraic Geometry II}, Ergeb. Math. Grenzgeb. (3), Vol. 49, Springer-Verlag, Berlin, 2004.
\bibitem{LJ-T} M. Lejeune-Jalabert, B. Teissier, \emph{Cl\^{o}ture int\'{e}grale des id\'{e}aux et \'{e}quisingularit\'{e}}, Ann. Fac. Sci. Toulouse Math. 17 (2008), 781--859.
\bibitem{Li_BS} Z. Q. Li, \emph{On the Brian\c{c}on-Skoda theorem for analytic local rings with singularities}, J. Algebra 577 (2021), 45--60.
\bibitem{Li_multiplier} Z. Q. Li, \emph{Nadel-Ohsawa multiplier ideal sheaves on complex spaces and singularities of plurisubharmonic functions}, preprint, \url{http://www.escience.cn/people/ZhenqianLI/index.html}. See also arXiv: 2003.11717.
\bibitem{Li_pullback} Z. Q. Li, \emph{A characterization of plurisubharmonic functions by holomorphic pull-backs (in Chinese)}, preprint, \url{http://www.escience.cn/people/ZhenqianLI/index.html}.
\bibitem{Li_BSOT} Z. Q. Li, \emph{From the Brian\c{c}on-Skoda theorem to $L^2$ division and extension theorems (in Chinese)}, preprint, \url{http://www.escience.cn/people/ZhenqianLI/index.html}.
\bibitem{L-Z} Z. Q. Li, X. Y. Zhou, \emph{Multiplier ideal sheaves on $\mathbb{Q}$-Gorenstein complex spaces}, preprint.
\bibitem{L-S} J. Lipman, A. Sathaye, \emph{Jacobian ideals and a theorem of Brian\c{c}on-Skoda}, Michigan Math. J. 28 (1981), 199--222.
\bibitem{L-T} J. Lipman, B. Teissier, \emph{Pseudo-rational local ring and a theorem of Brian\c{c}on-Skoda about integral closures of ideals}, Michigan Math. J. 28 (1981), 97--116.
\bibitem{Matsu_morphism} S.-I. Matsumura, \emph{Injectivity theorems with multiplier ideal sheaves for higher direct images under K\"ahler morphisms}, preprint, arXiv: 1607.05554.
\bibitem{Nara_Levi} R. Narasimhan, \emph{The Levi problem for complex spaces I $\&$ II}, Math. Ann. 142 (1961), 355--365; Math. Ann. 146 (1962), 195--216.
\bibitem{Ohsawa_book18} T. Ohsawa, \emph{$L^2$ Approaches in Several Complex Variables---Towards the Oka-Cartan Theory with Precise Bounds}, 2nd Edition, Springer, Tokyo, 2018.
\bibitem{Richberg68} R. Richberg, \emph{Stetige streng pseudokonvexe Funktionen}, Math. Ann. 175 (1968), 257--286.
\bibitem{Shibata17} K. Shibata, \emph{Rational Singularities, $\omega$-Multiplier Ideals, and Cores of Ideals}, Michigan Math. J. 66 (2017), 309--346.
\bibitem{SH06} I. Swanson, C. Huneke, \emph{Integral Closure of Ideals, Rings, and Modules}, London Math. Soc. Lecture Note Series, Vol. 336, Cambridge University Press, Cambridge, 2006.
\bibitem{Szna10} J. Sznajdman, \emph{An elementary proof of the Brian\c{c}on-Skoda theorem}. Ann. Fac. Sci. Toulouse Math. 19 (2010), 675--685.

\end{thebibliography}
\end{document}